\documentclass[a4paper,10pt]{article}

\usepackage{fullpage}
\usepackage{ae,lmodern}
\usepackage[english,greek,frenchb,french]{babel}
\usepackage[utf8]{inputenc}  
\usepackage[T1]{fontenc}
\usepackage{url,csquotes}
\usepackage{float}
\usepackage{amsfonts,amssymb,enumerate}
\usepackage{amsmath}
\usepackage{authblk}
\usepackage[shortlabels]{enumitem}
\allowdisplaybreaks[1]
\usepackage{amsthm}
\usepackage{graphicx}
\usepackage{bbm}
\usepackage{listings}
\usepackage{titling}
\usepackage{dsfont}
\usepackage{color}
\usepackage{enumitem}
\usepackage{bmpsize}
\usepackage{multibib}
\newcites{perso}{Publications and preprints}
\usepackage[hidelinks,hyperfootnotes=false]{hyperref}
\newcommand{\E}{\mathbb{E}}

    \newcommand{\Prb}{\mathbb{P}}
	\newcommand{\cR}{\mathcal{R}}
		
			\newcommand{\cD}{\mathcal{D}}
		
		\newcommand{\cE}{\mathcal{E}}
		\newcommand{\cF}{\mathcal{F}}
		\newcommand{\cG}{\mathcal{G}}

	\newcommand{\sR}{\mathbb{R}}

				\DeclareMathOperator{\Var}{Var}
				\DeclareMathOperator{\Ent}{Ent}
				
				\DeclareMathOperator{\shell}{shell}

    \newcommand{\sZ}{\mathbb{Z}}
    \newcommand{\sC}{\mathcal{C}}
    \newcommand{\ep}{\varepsilon}

    \newcommand{\ind}{\mathds{1}}

 \theoremstyle{plain}   
 \newtheorem{thm}{Theorem}[section]
\newtheorem{lem}[thm]{Lemma}

\newtheorem{rk}[thm]{Remark}

\numberwithin{equation}{section}

\title{The variance of the graph distance in the infinite cluster of percolation is sublinear}
\date{}
\author{Barbara Dembin\thanks{ETH Z\"urich}}
\begin{document}
\selectlanguage{english}
\maketitle
\begin{abstract}
    We consider the standard model of i.i.d. bond percolation on $\sZ^d$ of parameter $p$. When $p>p_c$, there exists almost surely a unique infinite cluster $\sC_p$. Using the recent techniques of Cerf and Dembin \cite{CerfDembinreg}, we prove that the variance of the graph distance in $\sC_p$ between two points of $\sC_p$ is sublinear. The main result extends the works of Benjamini, Kalai and Schramm \cite{BKS}, Benaim and Rossignol \cite{BenRos2008} and Damron, Hanson and Sosoe \cite{Damron2015} for the study of the variance of passage times in first passage percolation without moment conditions on the edge-weight distribution.
\end{abstract}
\section{Introduction}
{\textbf {Percolation.}}
Let $d\geq 2$. We consider an i.i.d. supercritical bond percolation on $\sZ^d$, every edge is open with a probability $p>p_c$, where $p_c$ denotes the critical parameter for this percolation. Let $\cG_p$ be the graph of the open edges
\[\cG_p:=(\sZ^d,\{e\in E^d: \text{ $e$ is open}\})\,.\]
We know that there exists almost surely in $\cG_p$ a unique infinite open cluster $\sC_p$ \cite{grimmettt:percolation}. 
We denote by $\cD^{\sC_p}$ the graph distance in the cluster $\sC_p$ that is
\begin{equation}
\forall x,y\in\sZ^d\qquad \cD^{\sC_p}(x,y):=\inf\Big\{\,|r|: \text{$r$ is a path from $x$ to $y$ in $\sC_p$}\,\Big\}\,
\end{equation}
where $|r|$ denote the number of edges in the path $r$ and we use the convention that $\inf\emptyset=+\infty$.
 In particular, if $x$ and $y$ are not connected in $\sC_p$, then we have $\cD^{\sC_p}(x,y)=\infty$. 
 To deal with the fact that $\cD^{\sC_p}(x,y)$ is infinite with positive probability, we will use the technique of Cerf and Théret in \cite{cerf2016} and introduce regularized points.
For $x$ in $\sZ^d$, we define $\widetilde x$ to be the closest point in $\sC_p$ to $x$ with a deterministic rule to break ties.  The advantage of defining regularized points is that for any $x,y\in\sZ^d$, $\cD^{\sC_p}(\widetilde x,\widetilde y)<\infty$ almost surely.

{\textbf {First passage percolation.}} The model of first passage percolation may be seen as a generalization of the model of percolation. Let $G$ be a distribution on $\sR_+\cup\{+\infty\}$.
  To each edge $e\in\E^d$, we assign a random variable $t_e$ such that the family $(t_e,\,e\in\E^d)$ is independent and identically distributed with distribution $G$. The random variable $t_e$ may be interpreted as the time needed to cross the edge $e$. We define a random pseudo-metric $T$ on this graph: for any pair of vertices $x$, $y\in\sZ^d$, the random variable $T(x,y)$ is the shortest time to go from $x$ to $y$, \textit{i.e.}, 
\[T(x,y):=\inf\left\{\sum_{e\in r}t_e\,:\text{ $r$ is a path joining $x$ to $y$}\right\}\,.\]
Note that for the distribution 
\begin{equation}\label{def:Gp}
    G_p=p\delta_1+(1-p)\delta_\infty,\quad p>p_c
\end{equation}
the travel time $T(x,y)$ for the law $G_p$ coincides with the graph distance between $x$ and $y$ in $\cG_p$ where the edges with infinite passage time correspond to the closed edges. 
Thanks to classical tools used in first passage percolation, in particular the subadditive ergodic theorem, we can study $\cD^{\sC_p}(\widetilde 0, \widetilde{nx})$. In particular, Cerf and Théret proved in \cite{cerf2016} a law of large numbers: there exists a deterministic function $\mu_p:\sZ^d\rightarrow [0,+\infty)$ such that
\begin{align*}
\forall x\in\sZ^d\qquad\lim_{n\rightarrow \infty} \frac{\cD^{\sC_{p}}(\widetilde{0},\widetilde{nx})}{n}=\mu_p(x)\qquad\text{a.s. and in $L^1$.}
\end{align*}
The function $\mu_p$ is the so-called time constant.

{\textbf {Fluctuations of the travel time.}}
The question of the fluctuations of $T(0,x)$ for general distributions $G$ is a very central question. It has been conjectured by physicists that the variance $\Var(T(0,x))$ should scale as $\|x\|_1^\alpha$ for some constant $\alpha<1$ depending on the dimension. In particular, in dimension $2$, it is conjectured that the model belongs to the KPZ universality class that was introduced by Kardar, Parisi and Zhang \cite{KPZ}
in 1986, and that $\alpha= 2/3$. However, beyond some related integrable models, the results obtained in this direction are still very modest.
The first upper-bound on the variance was obtained by Kesten in \cite{Kesten93}. He proved that there exists a constant $C$ such that for all $x\in\sZ^d$
\[ \Var(T(0,x))\le C\|x\|_1\]
under some integrability condition on the distribution $G$.
In their seminal paper \cite{BKS}, Benjamini, Kalai and Schramm proved that the fluctuations are sublinear there exists a constant $C$ such that for all $x\in\sZ^d$
\[ \Var(T(0,x))\le C\frac{\|x\|_1}{\log\|x\|}\]
for the case of a distribution $G$ that takes only two values.
The results of \cite {BKS} were later extended to continuous distributions that satisfied a modified logarithmic Sobolev inequality by Benaim and Rossignol \cite{BenRos2008} and to more general distributions under moment conditions using a Bernoulli encoding by Damron, Hanson and Sosoe \cite{Damron2015}.
In this paper, we aim to extend these results to the case of the distribution $G_p$ that was defined in \eqref{def:Gp}.
Namely, we are interested in the variance of the graph distance $\cD^{\sC_{p}}(\widetilde{0},\widetilde{nx})$. The main obstacle to overcome is that the distribution $G_p$ has no moment.
We obtain that the variance of the graph distance in $\sC_p$ is sublinear.
\begin{thm}\label{thm:main}Let $p>p_c$. There exists a positive constant $C_0$ depending on $p$ and $d$ such that 
\begin{equation}
\forall n \geq 1\quad\forall x\in\sZ^d\qquad \Var(\cD^{\sC_{p}}(\widetilde{0},\widetilde{nx}))\leq C_0\frac{n}{\log n}\,.
\end{equation}

\end{thm}

Let us first explain the proof strategy used in \cite{BKS}.
To prove that the variance of $T(0,nx)$ is sublinear, Benjamini, Kalai and Schramm use an inequality of Talagrand \cite{talagrand} related to hypercontractivity. They need to study the expected impact of changing the value of a given edge, that is called the influence of an edge. To use the result of Talagrand, they need that almost all the edges have a small influence. In this context, the influence of an edge $e$ is related to the probability that the geodesic $\gamma$ between $0$ and $nx$ goes through the edge $e$. Since there exists no result controlling the probability that the geodesic goes through a given edge, the authors use a trick to circumvent this issue. They randomized the starting point of the geodesic in such a way that the new random variable has a variance that is still close to the original one and such that all the edges of the lattice have a small influence. This trick of randomizing the starting point was later replaced by Damron \textit{et al.} by a geometric average in \cite{Damron2015}.

In \cite{BKS,BenRos2008,Damron2015}, moment conditions on the distribution are needed. A first reason why a moment condition is needed, is that without it, we have $\Var(x,y)=+\infty$. Note that even if we don't have a good moment condition, this problem may be solved by the use of regularized points. This is exactly in the same spirit than the use of regularized points for the study of the graph distance ensures that the graph distance is finite and has good moment properties.
But, the main reason why moment conditions are needed, is that it enables to obtain a good control on the impact of resampling an edge.
When the distribution is bounded, resampling an edge on the geodesic cannot affect too much the graph distance between two points. We can easily upperbound the influence by a constant. However, in the context of the graph distance in the infinite cluster of percolation, closing one edge on the geodesic can have a big impact on the graph distance. This is the main issue to extend the previous results to the distribution $G_p$ that can take infinite value. To solve this issue, we use here the recent technology developed by Cerf and Dembin in \cite{CerfDembinreg}. Before stating the key ingredient of the proof, let us introduce some definitions. Let $p>p_c$.
 Let $\sC^ {e}_p$ be the infinite connected component of $\sC_p\setminus \{e\}$, it is almost surely unique. For $x\in\sZ^d$, denote by $\widetilde x ^ e $ the closest point to $x$ in $\sC_p^ e$. Let us denote by $\cR_e$ the following event
 \begin{equation}\label{eq:defFe}
    \cR_e:=\{\widetilde 0=\widetilde 0 ^ e,\widetilde{nx}=\widetilde{nx}^ e\}\,. 
 \end{equation}
 The event $\cR_e$ is the event that the regularized points are unchanged when closing the edge $e$.
 We will need the following theorem that is the key result to prove the main theorem.
\begin{thm}\label{thm:cont}Let $p>p_c$. There exists a positive constant $c_0$ depending on $p$ and $d$ such that for any $x\in\sZ^d$ and $n\ge 1$, let $\gamma$ be a geodesic from $\widetilde 0$ to $\widetilde{nx}$, we have
\begin{equation*}
    \E\left[\sum_{e\in\gamma}(\cD^{\sC_{p}\setminus\{e\}}(\widetilde{0},\widetilde{nx})- \cD^{\sC_{p}}(\widetilde{0},\widetilde{nx}))^2\ind_{\cR_e}\right]\le c_0 n\,.
\end{equation*}

%\begin{equation}
%\Prb\left(\sum_{e\in\gamma}c(e)^2\geq c_0n\right)\leq \exp(-n^{1/6})\,.
%\end{equation}
\end{thm}

 Roughly speaking, this result says that on average, closing an edge on the geodesic modifies the graph distance by at most a constant. This theorem is a consequence of the work of Cerf and Dembin in \cite{CerfDembinreg}. This theorem together with the Efron-Stein inequality leads to an upper-bound on the variance of order $n$ which is already a new result in the context of the graph distance (with some additional technical details due to the use of regularized points). To prove that the variance is sublinear we will use the geometric averaging trick and concentration inequalities used by Damron and al in \cite{Damron2015}. This geometric average will ensure that every edge in the lattice has a small influence. Once the key result Theorem \ref{thm:cont} is proved,  the remaining of the proof used the concentration inequalities in the same way as \cite{Damron2015} with some additional technical difficulties due to the fact that we use regularized points. 
 \begin{rk}
 We believe that our proof strategy together with the Bernoulli encoding used by Damron and al in \cite{Damron2015} can also work for any distribution $G$ on $\sR_+\cup\{+\infty\}$ such that $G(\{+\infty\})<1- p_c$.
 \end{rk}
 
 In Section \ref{sec:background}, we present some standard facts about supercritical percolation and we present the concentration inequalities we will use. In Section \ref{sec:proofs}, we prove Theorems \ref{thm:main} and \ref{thm:cont}.
 \section{Background}\label{sec:background}
 \subsection{Background on percolation}
 We will need the following standard facts about percolation.
 For $x\in \sZ^d$, let $\sC_p(x)$ be the $p$-open cluster of $x$. We denote by $\|\cdot\|_2$ the $\ell_2$ norm. We have the following theorem that controls the probability of having a large and finite open cluster.
 \begin{thm}[Theorems 8.18 and 8.19 in \cite{grimmettt:percolation}]\label{thm:nosmallclus} Let $p>p_c$. There exist positive constants $A_1$ and $A_2$ such that 
 \[\forall n\ge 1\qquad\Prb(0\notin \sC_p, \sC_p(0)\cap \partial \Lambda_n\ne\emptyset)\le A_1\exp(-A_2 n)\,\]
 where $\Lambda_n=[-n,n]^ d\cap \sZ^ d$ and $\partial\Lambda_n=\{y\notin\Lambda_n:\exists x\in\Lambda_n, \{x,y\}\in\E^ d\}$.
 \end{thm}
 The following theorem controls the probability of having a big hole in the infinite cluster.
 \begin{thm}[Theorem 7 in \cite{grimmettt:percolation}]\label{thm:hole}Let $p>p_c$. There exist positive constants $A_3$ and $A_4$ such that 
 \[\forall n\ge 1\qquad\Prb(\sC_p\cap \Lambda_n =\emptyset)\le A_3\exp(-A_4 n)\,.\]
 \end{thm}
 The following theorem gives a control on the graph distance.
 \begin{thm}[Antal and Pisztora \cite{AntalPisztora}] \label{thm:antal}
 Let $p>p_c$. There exist positive constants $\beta$, $A_5$ and $A_6$ such that
 \[\forall x,y\in\sZ^d\quad \forall m\ge \beta \|x-y\|_2\qquad\Prb(m\le \cD^{\sC_p}(x,y)<\infty)\le A_5\exp(-A_6m)\,.\]
 \end{thm}
 We will need in what follows the two following estimates that are consequences of Theorem \ref{thm:antal}.
\begin{lem}\label{lem:contesp}Let $p>p_c$. There exists a constant $\kappa>0$ such that 
\[\forall k\ge 2\quad\forall x,y\in\mathbb Z^d\qquad\E[\cD^{\sC_p}(\widetilde x,\widetilde y)^k]\leq \log ^{2k}n +\kappa \|x-y\|_2^k\,.\]
\end{lem}
The following lemma controls the expected intersection of a geodesic with a box.
\begin{lem}\label{lem:interbox}Let $p>p_c$. There exists $\alpha>0$ such that for $n\ge 1$ and $x\in\sZ^d$. Let $\gamma$ be a geodesic from $\widetilde 0$ and $\widetilde nx$ in $\sC_p$, we have
\begin{equation*}
    \forall z\in\sZ^d\quad \forall m\ge \log ^2n \qquad \E|\gamma\cap (z+\Lambda_m)|\le \alpha m\,.
\end{equation*}
\end{lem}
Let us now prove these two lemmas.
\begin{proof}[Proof of Lemma \ref{lem:contesp}]
Set 
$l=\|\widetilde x-x\|_2+\|\widetilde y-y\|_2$.
Let $m\ge 2 \beta\|x-y\|_2$, we have using Theorems \ref{thm:hole} and \ref{thm:antal}
\begin{equation*}
    \begin{split}
        \Prb( \cD^{\sC_p}(\widetilde x,\widetilde y)\ge m)&\le \Prb\left(l \ge \frac{m}{4d\beta}\right)+\Prb(\exists w\in (x+\Lambda_{\frac{m}{4d\beta}})\,\exists z\in (y+\Lambda_{\frac{m}{4d\beta}}): m\le \cD^{\sC_p}(w,z)<\infty)\\
        &\le 2A_3\exp\left(-A_4\frac{m}{4d\beta}\right)+2m^dA_5\exp(-A_6m)
    \end{split}
\end{equation*}
where we use that for any $ w\in (x+\Lambda_{m/(4d\beta)})$ and $z\in (y+\Lambda_{m/(4d \beta )})$
\[\beta\|w-z\|_2 \le \frac{m}{2}+\beta\|x-y\|_ 2 \le m\,.\]
Hence, it yields that
for $k\ge 2$,
\begin{equation*}
    \E[\cD^{\sC_p}(\widetilde x,\widetilde y)^k]\le\max (\log ^2n, 2\beta \|x-y\|_2) ^{k} + \sum_{j\ge \max (\log ^2n, 2\beta \|x-y\|_2) }j^k \Prb(\cD^{\sC_p}(\widetilde x,\widetilde y)\ge j)
\end{equation*}
and the result follows.
\end{proof}

\begin{proof}[Proof of Lemma \ref{lem:interbox}] Let $x\in \sZ^d$, $n\ge 1$. Let $\gamma$ be the geodesic between $\widetilde {0}$ and $\widetilde {nx}$.
Let $z\in\sZ^d$ and $m\ge \log ^2n$. Let us assume $\gamma\cap (z+\Lambda_m)\ne\emptyset$, let us denote by $w$ and $y$ the first and last intersection of $\gamma$ with $z+\Lambda_m$. The portion of $\gamma $ between $w$ and $y$ is a geodesic, its length is equal to $\mathcal D ^{\sC_p}(w,y)$.
Let us denote by $\cE$ the following event
\[\cE:= \left\{\forall x,y\in (z+\partial \Lambda_m)\cap\sC_p: \mathcal D^{\sC_p}(x,y)\le 2d\beta m\right\}\,.\]
Thanks to Theorem \ref{thm:antal}, we have
\[\Prb(\cE^c)\le (2^ddm^{d-1})^2A_5\exp(-A_6m)\,.\]
Hence, we have
\begin{equation*}
    \begin{split}
        \E|\gamma\cap (z+\Lambda_m)|&\leq  \E[|\gamma\cap (z+\Lambda_m)|\ind_{\cE}]+\E[|\gamma\cap (z+\Lambda_m)|\ind_{\cE^c}]\\
        &\le 2d\beta m+2^{3d}m^{3d}A_5\exp(-A_6m)\,.
    \end{split}
\end{equation*}
 The result follows.

\end{proof}

\subsection{Concentration inequalities}
Let $f$ be a real-valued function on $\{0,1\}^{\E^d}$.
Let us enumerate the edges of the lattice $\{e_1,e_2,\dots\}$, we can write the following martingale decomposition
\begin{equation*}
f-\E f=\sum_{k=1}^\infty E(f|\cF_k)-\E(f|\cF_{k-1})
\end{equation*}
where $\cF_k$ is the $\sigma$-algebra generated by the first $k$ edge weights $t_{e_1},\dots,t_{e_k}$.
Set \[V_k=\E(f|\cF_k)-\E(f|\cF_{k-1})\,.\]
We have the following inequality that is proved in \cite{Damron2015} using an inequality proved by Falik and Samorodnitsky \cite{falik_samorodnitsky_2007}.
\begin{lem}[Lemma 3.3. in \cite{Damron2015}]\label{lem:ineqsob}
\begin{equation}
\Var(f)\log \left(\frac{\Var(f)}{\sum_{k=1}^\infty\E(|V_k|)^2}\right)\le \sum_{k=1}^\infty \Ent(V_k^2)
\end{equation}
where $\Ent$ denotes the entropy.
\end{lem}
We recall that the entropy of a function $f\ge0$ on $\{0,1\}^{\E^d}$ given a distribution $\pi$ on $\{0,1\}^{\E^d}$ is defined as follows
\[\Ent_\pi(f):=\E_\pi\left[f\log\frac{f}{\E_\pi(f)}\right]\,.\]
Let $p\in(0,1)$.
We will omit $\pi$ when $\pi$ is $\otimes_{e\in \E^d}\textrm{Ber}(p)$ where $\textrm{Ber}(p)$ denotes the Bernoulli distribution of parameter $p$.
For $e\in \E^d$, we define $\sigma_e ^1$ (respectively $\sigma_e^0$ the function that for $t\in\{0,1\}^{\E^d}$ changes $t_e$ into $1$ (respectively into $0$).
Set \[\Delta_e f:=f\circ \sigma_e ^0-f\circ \sigma_e ^1\,.\]
\begin{rk}The edges with value $0$ will correspond to the closed edges and the edges with value $1$ to the open edges. For $f$ a non-increasing functions on $\{0,1\}^{\E^d}$ we have $\Delta_e f \ge 0$. 
\end{rk}
We will need the following lemma to control the entropy.
\begin{lem}\label{lem:logsobber} Let $f$ such that $\E[f^4]<\infty$.
There exists a constant $C>0$ depending on $p$ such that
\begin{equation}
\sum_{k=1}^\infty\Ent(V_k^2)\le C\sum_{k=1}^\infty\E [(\Delta_{e_k} f) ^2]\,.
\end{equation}
\end{lem}
For short, we will write $\Delta_k$ instead of $\Delta_{e_k}$.
The proof of this lemma follows the same idea as in Lemma 6.3 in \cite{Damron2015} but in a simpler context. For sake of completeness, we include the proof of this lemma here. To prove this lemma we need the following lemma.
\begin{lem}[Bernoulli log-Sobolev inequalities]\label{lem:lsibernoulli} Let $p\in(0,1)$. There exists a positive constant $C$ depending on $p$ such that for any function $g\ge 0$ on $\{0,1\}$
\[\Ent_{\mathrm{Ber}(p)}[g^2]\le C(g(0)-g(1))^2\,.\]
\end{lem}
We will also need the following theorem.
\begin{thm} [Tensorization of the entropy, Theorem 2.3 in \cite{Damron2015}]\label{thm:tens}Let $p\in(0,1)$. Let $f$ be a non-negative $L^2$ random variable on $\{0,1\} ^{\E^d}$. Let $(t_e)_{e\in\E^d}$ be a family of i.i.d. Bernoulli random variable of parameter $p$ and denote $\pi$ the distribution of the family. For $t\in\{0,1\} ^{\E^d}$, denote by $\pi_k(t)$ be the distribution with respect to the $k^{th}$ coordinate, all the other coordinates remain fix.
We have
\[\Ent_\pi(f)\le \sum_{k=1}^\infty\E_\pi[\Ent_{\pi_k}(f)]\,.\]

\end{thm}
We have now all the ingredients to prove Lemma \ref{lem:logsobber}.
\begin{proof}[Proof of Lemma \ref{lem:logsobber}] Let $k\ge 1$. The random variable $V_k$ only depends on $t_{e_1},\dots,t_{e_k}$. 
Using Theorem \ref{thm:tens} and Lemma \ref{lem:lsibernoulli}, we have
\begin{equation*}
    \Ent_\pi(V_k^2)\le \sum_{j=1}^k \E_\pi[\Ent_{\pi_j}(V_k^2)]\le C\sum_{j=1}^k\E_\pi [(\Delta_j V_k) ^2] 
\end{equation*}
%Let $(t'_e)_{e\in E^d}$ be an i.i.d. family of Bernoulli of parameter $p$.
%Write $\pi'$ the law of $(t'_e)_{e\in E^d}$.
%We can write
%\[V_k=\E_{\pi'}[f(t_1,\dots,t_k,t'_{k+1},\dots)- f(t_1,\dots,t_{k-1},t'_{k},\dots)]\,.\]
It follows that 
\begin{equation*}
   \E[ (\Delta_j V_k)^2]=\left\{\begin{array}{ll}\E[\E[(\Delta_k f)|\cF_k]^2]&\mbox{if $j=k$}\\
   \E[ (\E[\Delta_j f|\cF_k]- \E[\Delta_j f|\cF_{k-1}])^2]&\mbox{if $j<k$}\,.
   \end{array}\right.
\end{equation*}
Hence
\begin{equation*}
    \begin{split}
      \sum_{k=1}^\infty\sum_{j=1}^k\E [(\Delta_j V_k) ^2]&=\sum_{j=1}^\infty\left(\E[\E[(\Delta_j f)|\cF_j]^2]+ \sum_{k\ge j+1} \E[ (\E[\Delta_j f|\cF_k]- \E[\Delta_j f|\cF_{k-1}])^2]\right)\\
&=\sum_{j=1}^\infty \lim_{N\rightarrow \infty}\E[\E[\Delta_jf|\cF_N]^2]  =\sum_{j=1}^\infty \E[(\Delta_j f)^2]
    \end{split}
\end{equation*}
where we used the orthogonality of the martingale increments in $L^2$ and the convergence of closed martingales. The result follows.
\end{proof}

\section{Proofs}\label{sec:proofs}
\subsection{Proof of Theorem \ref{thm:main}}
Let $p>p_c$. Let $(B_e)_{e\in E^d}$ be an i.i.d. family of Bernoulli random variable of parameter $p$.
Let $n\ge 1$ and $x\in\sZ^d$.
Set $m$ be the largest integer such that $m \le n^{1/4}$.
 Set $\Lambda_m=[-m,m]^d\cap\sZ^d$ and
\begin{equation*}
f((B_e)_{e\in\E^d}):=\frac{1}{|\Lambda_m|}\sum_{z\in \Lambda_m}\cD^{\sC_p}(\widetilde z,\widetilde {nx+z})
\end{equation*}
where edges that have value $1$ correspond to edges that are open. The function $f$ is a geometric average of the graph distance, the interest of considering such a function is that it is simpler to prove that all the edges have a small influence.
We can prove that the variance of $f$ is close to the original variance we aim to estimate. This is the purpose of the following lemma.
\begin{lem} \label{lem:compvar} We have for $n$ large enough
\begin{equation*}
\Var\cD^{\sC_p}(\widetilde 0,\widetilde {nx})\le 2\Var(f) +  n^{3/4}\,.
\end{equation*}
\end{lem}
Let $e\in \E^d$.
We recall that $\sC^ e_p$ is the infinite connected component of $\sC_p\setminus \{e\}$. For $z\in\sZ^d$, denote by $\widetilde z ^ e $ the closest point to $z$ in $\sC^ e_p$.
For short, set
\begin{equation}\label{eq:defle}
    \ell(e):=\cD^{\sC_{p}^ e}(\widetilde{0}^ e,\widetilde{nx}^ e)- \cD^{\sC_{p}}(\widetilde{0},\widetilde{nx})\,.
\end{equation}
We will need the two following lemmas that give an upper-bound on $\ell(e)$.
\begin{lem}\label{lem:contl} There exists $\kappa_1$ depending only on $d$ such that for any $e\in\E^d$, for all $k\ge 1$
\[\E[\ell (e)^k\ind_{e\in \gamma}\ind_{\cR_e}]\le \kappa_1\log^{2k} n\]
where $\gamma$ is the geodesic between $\widetilde 0$ and $\widetilde{nx}$ in $\sC_p$.
\end{lem}
The following lemma upperbounds the total influence of the edges that change the regularized points.
\begin{lem}\label{lem:contsum} We have for $n$ large enough
\[\sum_{k=1}^{\infty}\E[\ell(e_k)^2\ind_{\cR_{e_k}^c}]\le \log^{6d} n\,.\]
\end{lem}

We will also need the following lemma.
\begin{lem}\label{lem:controlV_k} We have for $n$ large enough
\begin{equation*}
\sum_{k=1}^\infty\E(|V_k|)^2\le n^{15/16}.
\end{equation*}
\end{lem}
Before proving all these lemmas, let us show how they imply Theorem \ref{thm:main}.
\begin{proof}[Proof of Theorem \ref{thm:main}]
If $\Var(f)\le n^{31/32}$, then thanks to Lemma \ref{lem:compvar}, the result follows.
Otherwise, we have thanks to Lemma \ref{lem:ineqsob}
\begin{equation*}
\Var(f)\log \left(\frac{n^{31/32}}{\sum_{k=1}^\infty\E(|V_k|)^2}\right)\le \sum_{k=1}^\infty \Ent(V_k^2)
\end{equation*}
and using Lemma \ref{lem:controlV_k} we get
\begin{equation}\label{eq:eqvar}
\Var(f)\le\frac{32}{\log n} \sum_{k=1}^\infty \Ent(V_k^2).
\end{equation}
It is easy to check thanks to Lemma \ref{lem:contesp} that $\E(f^4)<\infty$.
Finally using Lemma \ref{lem:logsobber}, we get
\begin{align*}
\sum_{k=1}^\infty \Ent(V_k^2)\le C\sum_{k=1}^\infty \E((\Delta_{e_k} f) ^2)&\le C\sum_{k=1}^\infty \frac{1}{|\Lambda_m|}\sum_{z\in \Lambda_m}\E((\Delta_{e_k} \cD^{\sC_p}(\widetilde z,\widetilde {nx+z}))^2)
\end{align*}
where we use Cauchy-Schwarz in the second inequality.
Using the invariance by translation in distribution, it follows that 
\begin{align*}
\sum_{k=1}^\infty \Ent(V_k^2)\le C\sum_{k=1}^\infty \E((\Delta_{e_k} \tau) ^2)
\end{align*}
where for short we write $\tau=\cD^{\sC_p}(\widetilde 0,\widetilde {nx})$. Let $e\in\E^d$.
Note that $\Delta_e \tau $ is independent of $B_e$, it follows that
\begin{equation*}
     \E((\Delta_e\tau)^2)=\frac{1}{p} \E((\Delta_e\tau)^2\ind_{B_e=1})\,.
\end{equation*}
Let us denote by $\gamma$ the geodesic between $\widetilde 0$ and $\widetilde {nx}$ (if there are several possible choices, we choose one according to a deterministic rule). We recall that $\cR_e$ was defined in \eqref{eq:defFe} as the event where closing the edge $e$ does not modify the regularized points. 
Let us assume that we are on the event $\cR_e$ and that $e$ is originally open and outside the geodesic $\gamma$, then closing $e$ has no impact on the geodesic and $\Delta_e\tau=0$. It yields that
\[(\Delta_e\tau)^2\ind_{B_e=1}\ind_{\cR_e}=(\Delta_e\tau)^2 \ind_{\cR_e}\ind_{e\in\gamma}\,.\]
Besides, we have 
\[\Delta_e \tau\ind_{B_e=1}=\ell(e)\ind_{B_e=1}\,.\]
Hence, we have
\begin{equation*}
\begin{split}
  \sum_{k=1}^\infty \Ent(V_k^2)&\le \frac{C}{p}\sum_{k=1}^\infty \E((\Delta_{e_k} \tau) ^2\ind_{\cR_{e_k}}\ind_{e_k\in\gamma})+\E((\Delta_{e_k} \tau) ^2\ind_{\cR_{e_k}^c}\ind_{B(e_k)=1})  \\
  &\le \frac{C}{p}
    \E\left[\sum_{e\in\gamma}(\cD^{\sC_{p}\setminus\{e\}}(\widetilde{0},\widetilde{nx})- \cD^{\sC_{p}}(\widetilde{0},\widetilde{nx}))^2\ind_{\cR_e}\right]+\frac{C}{p}\sum_{k=1}^\infty\E(\ell(e_k) ^2\ind_{\cR_{e_k}^c})\,.
\end{split}
\end{equation*}
Thanks to Theorem \ref{thm:cont} and Lemma \ref{lem:contsum}, we have for $n$ large enough
\begin{equation*}
\sum_{k=1}^\infty \Ent(V_k^2)\le 2c_0\frac{C}{p_c} n\,.
\end{equation*}
Using inequality \eqref{eq:eqvar} and Lemma \ref{lem:compvar}, the result follows.

\end{proof}
Let us now prove the lemmas.
\begin{proof}[Proof of Lemma \ref{lem:compvar}]
We have
\[|\cD^{\sC_p}(\widetilde 0,\widetilde {nx})-\cD^{\sC_p}(\widetilde z,\widetilde {z+nx})|\le \cD^{\sC_p}(\widetilde 0,\widetilde z)+ \cD^{\sC_p}(\widetilde {nx},\widetilde {z+nx})\,.\]
It is easy to check using the invariance by translation in distribution that
\[\E(f)=\E (\cD^{\sC_p}(\widetilde 0,\widetilde {nx}))\,.\]
We recall that $m\le n^{1/4}$.
It follows that
\begin{equation*}
    \begin{split}
        \Var( \cD^{\sC_p}(\widetilde 0,\widetilde {nx}))&=\E((\cD^{\sC_p}(\widetilde 0,\widetilde {nx})-\E (f))^2)\le 2\E((f-\cD^{\sC_p}(\widetilde 0,\widetilde {nx}))^2) + 2\Var (f)\\
        &\le \frac{4}{|\Lambda_m|}\sum_{z\in \Lambda_m}\E (\cD^{\sC_p}(\widetilde 0,\widetilde {z})^2)+ 2\Var(f)\\
        &\le 2\Var(f)+ 4d^2\kappa m^2\le 2\Var(f) + n^{3/4}
    \end{split}
\end{equation*}
where in the second to last inequality we use Cauchy Schwarz inequality and Lemma \ref{lem:contesp}.
The result follows.
\end{proof}
\begin{proof}[Proof of Lemma \ref{lem:controlV_k}] For short write $\tau_z=\cD^{\sC_p}(\widetilde z,\widetilde {z+nx}) $. We have
\begin{equation}\label{eq:contVk}
   \begin{split}
        \E[|V_k|]&\le \frac{1}{|\Lambda_m|}\sum_{z\in\Lambda_m}\E|\E(\tau_z|\cF_k)- \E(\tau_z|\cF_{k-1})|\\
    &\le \frac{2}{|\Lambda_m|}\sum_{z\in\Lambda_m}\E |\Delta_{e_k}\tau_z|  \\
    &= \frac{2}{p|\Lambda_m|}\sum_{z\in\Lambda_m}\E[ |\Delta_{e_k-z}\tau_0|\ind_{B(e_k-z)=1}]\\
    &=\frac{2}{p|\Lambda_m|}\sum_{z\in\Lambda_m}(\E [|\ell(e_k-z)|\ind_{\cR_{e_k-z}^c}] +\E [\ell(e_k-z)\ind_{\cR_{e_k-z}}\ind_{e_k-z\in \gamma}] 
   \end{split}
\end{equation}
where we used similar arguments than in the proof of Theorem \ref{thm:cont}, in particular that $\Delta_e \tau_z$ is independent of $B_e$.
Using Lemma \ref{lem:contsum}, we get
\begin{equation}\label{eq:contV2}
    \sum_{z\in\Lambda_m}\E [|\ell(e_k-z)|\ind_{\cR_{e_k-z}^c}] \le \sum_{k=1}^\infty \E[\ell(e_k)^2\ind_{\cR_{e_k}^c}]\le \log^{6d} n\,.
\end{equation}
Using Cauchy-Schwarz inequality and Lemma \ref{lem:contl}, we have
\begin{equation}\label{eq:contV3}
\begin{split}
    \sum_{z\in\Lambda_m}\E [\ell(e_k-z)\ind_{\cR_{e_k-z}}\ind_{e_k-z\in \gamma}] &\le \left( \sum_{z\in\Lambda_m}\E [\ell(e_k-z)^2\ind_{\cR_{e_k-z}}\ind_{e_k-z\in \gamma}]\right)^{1/2}\left( \sum_{z\in\Lambda_m}\E [\ind_{e_k-z\in \gamma}]\right)^{1/2}\\
    &\le \log ^2n \sqrt{\kappa_1|\Lambda_m|}\sqrt{\E|\gamma\cap (\Lambda_m + e_k)|}\,.
    \end{split}
\end{equation}
Using Lemma \ref{lem:interbox} and the inequalities \eqref{eq:contVk}, \eqref{eq:contV2} and \eqref{eq:contV3}, it follows that for $n$ large enough
\begin{equation*}
     \E[|V_k|]\le \frac{2^d}{p_c}\sqrt{\alpha\kappa_1}\log^2 n\, m ^{(1-d)/2}\,.
\end{equation*}
Besides, using Theorem \ref{thm:cont}, Lemma \ref{lem:contsum} and inequality, we have for $n$ large enough
\eqref{eq:contVk}
\begin{equation*}
    \begin{split}
        \sum_{k=1}^\infty \E[|V_k|]&\le  \frac{2}{p}\E \left[\sum_{e\in \gamma}(\cD^{\sC_{p}\setminus\{e\}}(\widetilde{0},\widetilde{nx})- \cD^{\sC_{p}}(\widetilde{0},\widetilde{nx}))^2\ind_{\cR_{e}}\right] + \frac{2}{p}\sum_{k=1}^\infty \E[|\ell(e_k)|\ind_{\cR_{e_k}^c}] \le \frac{4}{p_c} c_0 n\,.
    \end{split}
\end{equation*}
Finally, combining the two previous inequalities we get for some constant $C$ for $n$ large enough
\begin{equation*}
    \begin{split}
        \sum_{k=1}^\infty  \E[|V_k|]^2\le C\log^2 n\, m ^{(1-d)/2}\sum_{k=1}^\infty  \E[|V_k|]\le  \frac{4}{p}C c_0n\,\log^2 n\,  m ^{(1-d)/2}\le n^{15/16}
    \end{split}
\end{equation*}
where we use that $m^{(d-1)/2}\ge n ^{1/8}/2$ (we recall that $m$ is the largest integer such that $m\le n^{1/4}$).
\end{proof}

\subsection{Upper bounding $\ell(e)$}
Theorem \ref{thm:main} trivially holds for $p=1$. We here work for $p\in(p_c,1)$.
\begin{proof}[Proof of Lemma \ref{lem:contl}]
Let $e=\{w,z\}\in \E^d$ and $k\ge 1$.
We have
\begin{equation*}
    \begin{split}
     \E[\ell (e)^k\ind_{\cR_e}\ind_{e\in \gamma}]&\le \E[\cD^{\sC_p\setminus \{e\}}(w,z)^k\ind_{w,z\in\sC^ e_p}]\\&=\frac{1}{1-p}\E[\cD^{\sC_p\setminus \{e\}}(w,z)^k\ind_{w,z\in\sC^ e_p}\ind_{B_e=0}]\\
     &\le \frac{1}{1-p}\E[\cD^{\sC_p}(w,z)^k\ind_{w,z\in\sC_p}]\le \frac{2\log^{2k} n}{1-p}
     \end{split}
\end{equation*}
where we use that $\cD^{\sC_p\setminus \{e\}}(w,z)^k\ind_{w,z\in\sC^ e_p}$ is independent of $B_e$. In the last inequality we used Lemma \ref{lem:contesp}.

\end{proof}
\begin{proof}[Proof of Lemma \ref{lem:contsum}]
Let $e=\{w,z\}\in \E^d$. Let $\sC^ e_p$ be the infinite connected component of $\sC_p\setminus \{e\}$. For $y\in\sZ^d$, denote by $\sC^e_p(y)$ the connected component of $y$ in the graph $\sC_p(y)\setminus\{e\}$. 
The following estimate will be used several time in what follows
\begin{equation}\label{eq:contl4}
    \begin{split}
        \E[\ell (e) ^4]&=\E[(\cD^{\sC_{p}\setminus\{e\}}(\widetilde{0}^ e,\widetilde{nx}^ e)- \cD^{\sC_{p}}(\widetilde{0},\widetilde{nx}))^4]\\&\le 8( \E[\cD^{\sC_{p}\setminus\{e\}}(\widetilde{0}^ e,\widetilde{nx}^ e)^4] +\E[\cD^{\sC_{p}}(\widetilde{0},\widetilde{nx})^4])\\&\le 8\frac{2-p}{1-p}\E[\cD^{\sC_{p}}(\widetilde{0},\widetilde{nx}))^4]\le C_2n^4
    \end{split}
\end{equation}
where we use Lemma \ref{lem:contesp} in the last inequality and $C_2$ is a constant depending on $p$.
 Set $$l:= \|w\|_\infty \,.$$
 Note that if $\widetilde 0\ne\widetilde 0^ e$ then either $\sC^ e_p(w)$ is finite and contains $\widetilde 0$ or $\sC^ e_p(z)$ is finite and contains $\widetilde 0$. 
Let us compute the following probability
\begin{equation*}
    \begin{split}
        \Prb(\widetilde 0\neq\widetilde 0^ e)&\le
         \Prb(\widetilde 0\notin \Lambda_{l/2})+ \Prb(\widetilde 0\in \Lambda_{l/2},\widetilde 0\neq\widetilde 0^ e)\\
        &\le \Prb(\sC_p(0)\cap \Lambda_{l/2}=\emptyset)+ \Prb\left(w\notin \sC_p^ e,\sC^ e_p(z)\cap (w+\Lambda_{l/2})\ne\emptyset\right)+\Prb\left(z\notin \sC_p^ e,\sC^ e_p(z)\cap (z+\Lambda_{l/2})\ne\emptyset\right)\\&\le A_3\exp\left(-A_4 \frac{l}{2}\right)+ \frac{2 A_1}{1-p}\exp\left(-A_2\frac{l }{2}\right)
    \end{split}
\end{equation*}
where we used Theorems \ref{thm:hole} and \ref{thm:nosmallclus}.
It follows that using the previous inequality and inequality \eqref{eq:contl4}
\begin{equation*}
    \begin{split}
        \E[\ell(e)^2\ind_{\widetilde 0\neq\widetilde 0^ e}]&\le \sqrt{\E[\ell(e)^4]\Prb(\widetilde 0\neq\widetilde 0^ e)}\le \sqrt{2C_2 A_1 (1-p) ^{-1}}n^2\exp(-A_2 \|w\|_\infty /4)
    \end{split}
\end{equation*}
and similarly we have
\begin{equation*}
    \begin{split}
        \E[\ell(e)^2\ind_{\widetilde {nx}\neq\widetilde{ nx}^ e}]
        &\le \sqrt{2C_2 A_1 (1-p) ^{-1}}n^2\exp(-A_2 \|nx-w\|_\infty/4)\,.
    \end{split}
\end{equation*}
Let us assume that $\min( \|w\|_\infty,\|nx-w\|_\infty)\ge \log ^2 n$.
We have for $n$ large enough
\begin{equation*}
\begin{split}
    \E[\ell(e)^2\ind_{\cR_e^c}]&\le  \E[\ell(e)^2\ind_{\widetilde 0 \ne \widetilde 0^ e}]+\E[\ell(e)^2\ind_{ \widetilde {nx}\ne \widetilde{nx}^ e}]\\
    &\le \exp(-A_2(\min( \|w\|_\infty,\|nx-w\|_\infty)) ^{2}/4)\,.
    \end{split}
\end{equation*}
Let us now assume that $\min( \|w\|_\infty,\|nx-w\|_\infty)\le \log ^2 n$. We can assume that $\|w\|_\infty\le \log ^2 n$. The case where $\|nx-w\|_\infty\le \log  ^2 n$ can be treated similarly. On the event $\{\widetilde 0 \ne \widetilde 0^ e, \widetilde {nx}=\widetilde{nx}^ e\}$,
we must have $e\in\gamma$ otherwise it would contradict that $\widetilde 0 \ne \widetilde 0^ e$ ($\widetilde{nx}$ cannot be connected to $\widetilde 0$ in $\sC^ e_p$). Besides, we have either $\widetilde {nx} \in\sC_p^ e(w)$ or $\widetilde {nx} \in\sC_p^ e(z)$.
If $\widetilde {nx} \in\sC_p^ e(w)$, we have
\[\cD^{\sC^ e_p}(\widetilde 0 ^ e,\widetilde {nx})\le \cD^{\sC^ e_p}(\widetilde 0 ^ e,w)+ \cD^{\sC^ e_p}(w,\widetilde {nx})=\cD^{\sC^ e_p}(\widetilde 0 ^ e,w)+ \cD^{\sC_p}(w,\widetilde {nx})\,\]
and 
\[\cD^{\sC_p}(\widetilde 0 ,\widetilde {nx})=\cD^{\sC_p}(\widetilde 0 ,w)+\cD^{\sC_p}(w ,\widetilde {nx})\,.\]
Finally, we have
\[\ell(e)^2\ind_{\widetilde 0 \ne \widetilde 0^ e}\ind_{ \widetilde {nx}=\widetilde{nx}^ e}\le (\cD^{\sC^ e_p}(\widetilde 0 ^ e,w)- \cD^{\sC_p}(\widetilde 0 ,w))^2\ind_{\widetilde 0 ^ e\in\sC_p^ e(w)}+(\cD^{\sC^ e_p}(\widetilde 0 ^ e,z)- \cD^{\sC_p}(\widetilde 0 ,z))^2\ind_{\widetilde 0 ^ e\in\sC_p^ e(z)}\,.\]
Using Lemma \ref{lem:contesp} and similar arguments as in the proof of Lemma \ref{lem:contl}, it yields that
\begin{equation*}
\begin{split}
    \E[\ell(e)^2\ind_{\widetilde 0 \ne \widetilde 0^ e}\ind_{ \widetilde {nx}=\widetilde{nx}^ e}]&\le 2\E[\mathcal D^{\sC^ e_p}(\widetilde 0 ^ e,z)^2\ind_{\widetilde 0^ e \in\sC_p^ e(z)}]+2\E[\mathcal D^{\sC^ e_p}(\widetilde 0 ^ e,w)^2\ind_{\widetilde 0^ e \in\sC_p^ e(w)}]\\
    &+2\E[\mathcal D^{\sC_p}(\widetilde 0 ,w)^2\ind_{\widetilde 0^ e \in\sC_p^ e(w)}]+2\E[\mathcal D^{\sC_p}(\widetilde 0 ,z)^2\ind_{\widetilde 0^ e \in\sC_p^ e(z)}]\\
   & \le 2\E[\mathcal D^{\sC^ e_p}(\widetilde 0 ^ e,\widetilde z^ e)^2]+2\E[\mathcal D^{\sC^ e_p}(\widetilde 0 ^ e,\widetilde w ^ e)^2]+2\E[\mathcal D^{\sC_p}(\widetilde 0 ,\widetilde w)^2]+2\E[\mathcal D^{\sC_p}(\widetilde 0 ,\widetilde z)^2]\\
    &\le \frac{8\kappa_1\log  ^{4}n}{1-p}\,.
    \end{split}
\end{equation*}
As a result, we have for $n$ large enough
\begin{equation*}
    \E[\ell(e)^2\ind_{\cR_e^c}]\le  \E[\ell(e)^2\ind_{\widetilde nx \ne \widetilde nx^ e}]+ \E[\ell(e)^2\ind_{\widetilde 0 \ne \widetilde 0^ e}\ind_{ \widetilde {nx}=\widetilde{nx}^ e}]\le \frac{10\kappa_1\log ^4n}{1-p} \,.
\end{equation*}
Finally, we have
\begin{equation*}
    \sum_{e\in\E^d}\E[\ell(e)^2\ind_{\cR_e^c}]\le\sum_{e\in (\Lambda_{\log ^2 n }\cup (nx+\Lambda_{\log ^2 n }))}\frac{10\kappa_1\log ^4 n }{1-p}+ \sum_{j\ge \log ^2n }c_d j ^{d-1}\exp(-A_2 j /4)\,
\end{equation*}
where $c_d$ is a constant depending only on $d$.
It follows that for $n$ large enough, we have 
\begin{equation*}
    \sum_{e\in\E^d}\E[\ell(e)^2\ind_{\cR_e^c}]\le \log  ^{6d}n \,.
\end{equation*}
The result follows.
\end{proof}
\subsection{Proof of Theorem \ref{thm:cont}}
 Let $p>p_c$. Let $x\in \sZ^d$ and $n\ge 1$. Let $\gamma$ be the geodesic between $\widetilde 0$ and $\widetilde {nx} $.
In this section, we will use results of \cite{CerfDembinreg}. The Proposition 3.5 states that for any path we can associate to each edge in the bulk of the path a shell. A shell is a set of boxes with good connectivity properties surrounding the edge. It is not important to understand the precise definition of a shell. What is important to understand is that this shell, thanks to the good connectivity of the boxes of the shells, enables us to build a bypass of the edge $e$ in a neighborhood of the shell of well-controlled length. In particular, the Proposition 3.6 enables to bound the size of this bypass of the edge $e$ by $\kappa_0|\shell(e)|$. Knowing that there exists a bypass of length at most $\kappa_0|\shell(e)|$ enables us to upper-bound the number of extra edges we need to join $\widetilde 0$ and $\widetilde {nx}$ when we close the edge $e$.
Finally, the Proposition 4.5 enables us to have with very high probability a good control on the average size of the shells built in Proposition 3.5.
Let $\ep>0 $ small enough depending on $p$ and such that $p-\ep>p_c$. Apply Proposition 3.5 in \cite{CerfDembinreg} to $p-\ep,p$ and $\gamma$, there exists a family $(\shell(e),e\in \overline \gamma)$ where $\overline \gamma = \gamma \setminus ((\Lambda_{\mathrm N}+ \widetilde 0)\cup (\Lambda_{\mathrm N}+\widetilde {nx})$ where $\mathrm N$ is a random variable (that corresponds to $N_{M(\gamma)}$ in Proposition 3.5). 
\begin{rk}
Note that here we will build $(p-\ep)$ bypass, so we implicitly work here with a coupling of the bond percolation of parameter $p-\ep$ and $p$ in such a way that a $(p-\ep)$-open edge is also $p$-open. Actually, the proof of Proposition 3.5 still holds true when $p=q$ and the definition of a good box becomes simpler. To stick to the exact context of Proposition 3.5 we use $p-\ep$ and $p$. One should not care too much about the $p-\ep$ since what we obtain at the end is $p$-open bypass so we can forget about the $p-\ep$.
\end{rk}
Thanks to Proposition 3.6 in \cite{CerfDembinreg}
\begin{equation}
\forall e \in \overline \gamma \qquad (\cD^{\sC_{p}\setminus\{e\}}(\widetilde{0},\widetilde{nx})- \cD^{\sC_{p}}(\widetilde{0},\widetilde{nx}))\ind_{\cR_e}\leq \kappa_0 |\shell(e)|
\end{equation}
where $\kappa_0$ is a constant depending on $p$, $d$ and $\ep$.
We recall that $\ell(e)$ was defined in \eqref{eq:defle}.
Thanks to the control on the size of the family $(|\shell(e)|,e\in\overline \gamma)$ in Proposition 3.5 and the Proposition 4.5 in \cite{CerfDembinreg}, there exist positive constants $C_1$ and $C_2$ such that for $n$ large enough
\[\Prb( \cE_0^c)\le \exp(- C_2n^{1/(6d^2+1)})\,.\]
where 
\[\cE_0:=\left\{  \mathrm{N}\le n^{1/3d},\,\sum_{e\in\overline \gamma}|\shell (e)|^2\le C_1 n\right\}\,.\]
Finally, we have the following control
\begin{equation*}
    \begin{split}
       \E\left[ \sum_{e\in\gamma}(\cD^{\sC_{p}\setminus\{e\}}(\widetilde{0},\widetilde{nx})- \cD^{\sC_{p}}(\widetilde{0},\widetilde{nx}))^2\ind_{\cR_e}\right]&\le\E\left[ \sum_{e\in\gamma}\ell(e)^2\ind_{\cR_e}\ind_{\cE_0}\right]+\E\left[ \sum_{e\in\gamma}\ell(e)^2\ind_{\cR_e}\ind_{\cE_0^c}\right]\\
       &\le C_1\kappa_0^2 n + \E\left[ \sum_{e\in\gamma\cap (\Lambda_{2n^{1/3d}}\cup (\Lambda_{2n^{1/3d}}+nx))}\ell(e)^2\ind_{\cR_e}\right]\\
       &\hspace{1cm}+ \sum_{e\in\E^d}\E\left[ \ell(e)^4\ind_{\cR_e}\ind_{e\in\gamma}\right] ^{1/2}\sqrt{\Prb(\cE_0^c)}\,
    \end{split}
\end{equation*}
where we use Cauchy-Schwarz inequality in the last inequality.
Besides, we have by Cauchy-Schwarz inequality and Lemma \ref{lem:contl} that
\begin{equation*}
\begin{split}
    \sum_{e\in\E^d}\E\left[ \ell(e)^4\ind_{\cR_e}\ind_{e\in\gamma}\right] ^{1/2}&\le \sum_{e\in\E^d}\E[\ell(e)^8\ind_{\cR_e}\ind_{e\in\gamma}]^{1/4}\Prb(e\in\gamma)^{1/4}\le\kappa_1 \log^4 n\sum_{e\in\E^d}\Prb(e\in\gamma)^{1/4}\,.
    \end{split}
\end{equation*}
It is easy to check that the right hand side is at most polynomial in $n$ using for instance Theorem \ref{thm:antal}. It follows that the following quantity goes to $0$ when $n$ goes to infinity
\[\sum_{e\in\E^d}\E\left[ \ell(e)^4\ind_{\cR_e}\ind_{e\in\gamma}\right] ^{1/2}\sqrt{\Prb(\cE_0^c)}\,.\]
Thanks to Lemma \ref{lem:contl}, we have
\begin{equation*}
    \begin{split}
        \E\left[ \sum_{e\in\gamma\cap (\Lambda_{2n{^1/3d}}\cup (\Lambda_{2n{^1/3d}}+nx))}\ell(e)^2\ind_{\cR_e}\right]&=\sum_{e\in  (\Lambda_{2n{^1/3d}}\cup (\Lambda_{2n{^1/3d}}+nx))}\E[\ell(e)^2\ind_{e\in\gamma}\ind_{\cR_e}]\le C_d n^{1/3}\log^4 n
    \end{split}
\end{equation*}
where $C_d$ is a constant depending only on $d$.
Combining the previous inequalities, we get for $n$ large enough
\begin{equation*}
    \E\left[\sum_{e\in\gamma}(\cD^{\sC_{p}\setminus\{e\}}(\widetilde{0},\widetilde{nx})- \cD^{\sC_{p}}(\widetilde{0},\widetilde{nx}))^2\ind_{\cR_e}\right]\le 2C_1\kappa_0^2 n\,.
\end{equation*}
The result follows.

\subsection*{Acknowledgements}The author was partially funded by the SNF Grant 175505 and the ERC Starting Grant CRISP and is part of NCCR SwissMAP. 

\bibliographystyle{plain}
\bibliography{biblio2}
\end{document}